\documentclass{amsart}
\usepackage{amsfonts}
\usepackage{amsthm}
\usepackage{amsmath}
\usepackage{amsfonts}
\usepackage{latexsym}
\usepackage{amssymb}
\usepackage[latin1]{inputenc}
\usepackage{verbatim}

\newcommand{\U}{\mathcal{U}}

\newcommand{\C}{\mathcal{C}}

\newcommand{\RR}{\mathbb{R}}
\newcommand{\NN}{\mathbb{N}}

\newtheorem{fed}{Definition}[section]
\newtheorem{teo}[fed]{Theorem}
\newtheorem*{teo*}{Theorem}
\newtheorem{lem}[fed]{Lemma}
\newtheorem{cor}[fed]{Corollary}
\newtheorem{pro}[fed]{Proposition}
\theoremstyle{definition}
\newtheorem{rem}[fed]{Remark}

\newtheorem{conj}[fed]{Conjecture}

\def\cA{\mathcal{A}}

\def\cC{\mathcal{C}}
\def\cD{\mathcal{D}}

\def\cP{\mathcal{P}}

\DeclareMathOperator{\Tr}{Tr} \DeclareMathOperator{\tr}{tr}

\newcommand{\op}{L(\mathcal{H})}

\newcommand{\CC}{\mathbb{C}}

\newcommand{\mat}{\mathcal{M}_n(\mathbb{C})}
\newcommand{\matsa}{\mathcal{M}_n(\CC)^{sa}}
\newcommand{\matpos}{\mathcal{M}_n(\CC)^+}

\begin{document}

\title[NC Schur-Horn theorems and extended majorization]{Non-commutative Schur-Horn theorems and extended majorization
for Hermitian matrices}

\author[Pedro G. Massey]{Pedro G. Massey}

\dedicatory{  A Marina, con amor}
\date{}
 \thanks{ Partially supported by a pdf from PIMS, CONICET (PIP
4463/96), Universidad de La Plata (UNLP 11 X350) and ANPCYT
(PICT03-09521).}

\begin{abstract} Let $\mathcal A\subseteq \mat$ be a unital $*$-subalgebra of the algebra
$\mat$ of all $n\times n$ complex matrices and let $B$ be an
hermitian matrix. Let $\U_n(B)$ denote the unitary orbit of $B$ in
$\mat$ and let $\mathcal E_\mathcal A$ denote the trace preserving
conditional expectation onto $\mathcal A$. We give an spectral
characterization of the set $$ \mathcal E_\mathcal
A(\U_n(B))=\{\mathcal E_\mathcal A(U^* B \,U):\ U\in \mat,\ \text{unitary  matrix}\}.$$ We obtain a similar result for the contractive
orbit of a positive semi-definite matrix $B$. We then use these
results to extend the notions of majorization and submajorization
between self-adjoint matrices to spectral relations that come
together with extended (non-commutative) Schur-Horn type theorems.
\end{abstract}

\maketitle


\medskip


\section{Introduction}
The Schur-Horn theorem states (\cite{Horn54,Schur23}), roughly
speaking, that the necessary and sufficient conditions on two
vectors $x,\, y\in \RR^n$ for the existence of an hermitian matrix
$A$ with spectrum (counting multiplicities) $y$ and main diagonal
$x$ are a finite set of linear inequalities involving the entries of
$x$ and $y$. This result was the starting point for the work
of Konstant \cite{Kons} on actions of compact Lie groups that was
subsequently extended to torus actions on symplectic manifolds by
Atiyah \cite{Ati}, and
 Guillemin and Sternberg \cite{GS} independently. Recently, there
 has been interest in some geometric aspects of the original result
 of Schur and Horn \cite{leite} which turn out to have also implications in frame
 theory \cite{RM}. 
 
There have also been extensions of the Schur-Horn theorem to
infinite dimensions such as Neuman's work on approximate
diagonals of selfadjoint operators in $\op$, the work of Kadison
\cite{Kad0,Kad1} particularly on diagonals of projections in $\op$,
and the recent work of Arveson and Kadison \cite{arvkad} on
diagonals of trace class operators, where they also focus on a
possible extension of the Schur-Horn theorem to II$_1$ factors. A
weak version of the Arveson-Kadison conjecture is proved in
\cite{argmas}. Indeed, this exposition in strongly influenced by the
point of view of \cite{Kad0} and \cite{arvkad} of the Schur-Horn
theorem.

In \cite{LiPoon} C.K. Li and Y.T. Poon obtained an extension of the Schur-Horn theorem,
but in a different way. They found necessary and sufficient
spectral conditions on two $n\times n$ selfadjoint matrices $A,\, B$
for the existence of an $n\times n$ unitary matrix $U$ such that $A$
is the \emph{block} diagonal compression of $U^*BU$ with respect to
certain block decomposition of $U^*BU$. Notice that the Schur-Horn
theorem can be seen as a particular case of this problem, namely
when the block representation of $U^*BU$ is with respect to
1$\times$1 blocks. They showed that the
situation with these general block compressions is quite
different from that of the classical Schur-Horn theorem (see for example Proposition \ref{no conv} below). The
nature and the complexity of the necessary and sufficient spectral
conditions they found are related with Klyachko's
compatibility inequalities \cite{Klya}, which give necessary and
sufficient conditions on $(m+1)$ vectors $\lambda^i\in \RR^n$,
$0\leq i\leq m$ for the existence of $(m+1)$ $n\times n$ selfadjoint
matrices $A_i$ with spectrum $\lambda ^i$ for $0\leq i\leq m$ and
$A_0=A_1+\ldots +A_m$.

In this note we consider a systematic analysis of what we consider
non commutative Schur-Horn type theorems. These include the previous
work \cite{LiPoon} on block diagonal compressions of the unitary
orbit of an hermitian matrix, block diagonal compressions of the
contractive orbit of a positive semidefinite matrix (see Theorem
\ref{teo contrac sh}) and partial traces of the unitary orbit of
an hermitian matrix (see Theorem \ref{teo par trac}). Our approach
is based on the  work of Friedland \cite{fried} and Fulton
\cite{ful} that extend that of Klyachko \cite{Klya} on the spectrum
of the sum of hermitian operators. These results are unified in the
following theorem, which provides operator algebra versions of the
Schur-Horn theorem, in the sense of \cite{arvkad}. We use the
following notation: given $\lambda^i\in \RR^{d(i)}$ for $1\leq i\leq
m$ with $\sum_{i=1}^m d(i)=n$ then $[\lambda^i]_{i=1}^n\in \RR^n$
denotes the vector obtained by juxtaposition of the vectors
$\lambda^i$'s i.e. $\lambda=(\lambda^1_1,\ldots, \lambda^1_{d(1)},
\lambda^2_1,\ldots, \lambda^m_{d(m)}).$ See also sections
\ref{spreliminares} and \ref{noentiendo} for notations and
terminology.

\begin{teo*}[NC Schur-Horn]\label{mayo esp}
Let $\mathbf l=(d(i),c(i))_{i=1}^m\in (\NN^2)^m$ be such that
$\sum_{i=1}^md(i)\cdot c(i)=n$ and consider the unital
$*$-subalgebra $\mathcal A=\oplus_{i=1}^m \mathcal
M_{d(i)}(\CC)\otimes 1_{c(i)}\subseteq \mat$. Let $\mathcal
E_\mathcal A$ denote the trace preserving conditional expectation
onto $\mathcal A$.
\begin{enumerate}
\item\label{uit1} If $B\in \matsa$ then there exists $M_B(\mathcal A)\subset\RR^n$,
  that can be generated in terms of Klyachko's compatibility inequalities, $\mathbf l$ and $\lambda(B)$, such that
 \begin{equation*}\label{desc alg op mayo}
 \mathcal E_\mathcal A(\U_n(B))=\{ \oplus_{i=1}^m A_i\otimes 1_{c(i)}\in \mathcal A: \ [\lambda(A_i\otimes 1_{c(i)})]_{i=1}^m\in M_B(\mathcal A)\}.
 \end{equation*}
\item\label{uit2} If $B\in \matpos$ then there exists $M^w_B(\mathcal A)\subset(\RR_{\geq 0})^n$,
  that can be generated in terms of Klyachko's compatibility inequalities, $\mathbf l$ and $\lambda(B)$,
 such that
 \begin{equation*}\label{desc alg op submayo}
 \mathcal E_\mathcal A(\cC_n(B))=\{ \oplus_{i=1}^m A_i\otimes 1_{c(i)}\in \mathcal A:
  \ [\lambda(A_i\otimes 1_{c(i)})]_{i=1}^m\in M_B^w(\mathcal A)\}.
 \end{equation*} \end{enumerate}
\end{teo*}
\noindent If $\mathcal A$ is as in the statement of the
NC-Schur-Horn theorem above then $\mathcal E_\mathcal A$ can be
described as
 $$ \mathcal E_\mathcal A(X)= \oplus_{i=1}^m
\frac{1}{c(i)} \sum_{j=t(i)}^{t(i+1)-1} X_{j}\otimes 1_{c(i)} \, , \
\ \text{ with } \ \ \C_\cP(X)=\oplus_{i=1}^c\, X_i$$ where
$c=\sum_{i=1}^m c(i)$,
$k(\sum_{r=1}^ {i-1} c(r)+j)=d(i)$  for $1\leq i\leq m$ , $1
\leq j\leq c(i)$, $\cP=\{P_i \}_{i=1}^c$ is the system of coordinate (diagonal)
projections with rank$\,(P_i)$ $=k(i)$ for $1\leq i\leq c$ and $t(i)=\sum_{j=1}^{i-1}c(j)+1$ for $1\leq i\leq m$.

Notice that although the existence of the sets $M_B(\mathcal A)$ and
$M_B^w(\mathcal A)$ in the NC-Schur-Horn theorem is trivial, their
description is not. Actually, we think that one of the main points
of this note is to show a relation between Klyachko's compatibility
inequalities and the description of these sets. We point out that
in the special case $\mathbf l=(1,1)_{i=1}^n$ (and hence $\mathcal
A$ is the diagonal masa) using the reduction of the
complexity of Klyachko's inequalities obtained in \cite{LiPoon} the
Schur-Horn theorem is recovered in terms of majorization.

This finite dimensional operator algebra point of view is developed
to introduce an extension of majorization between selfadjoint
matrices as defined by Ando \cite{Ando} to that of extended
majorization between selfadjoint matrices. Since this last concept
involves some technical notions we postpone its detailed discussion
until section \ref{noentiendo}.

We also consider the relation of extended majorization with some
convex functionals. As in the case of usual majorization, the notion of extended majorization has
relations with ``signal processing" (\cite{JMD,dhi,RM}), but it seems that in this case
the word ``quantum" may be added. As an example of this last claim,
we obtain a result related with a conjecture posed by M.B. Ruskai and K. Audenaert in Quantum Information Theory (QIT).

\smallskip

\noindent {\bf Acknowledgments.} This note is the consequence of a
talk I gave in the Canadian Operator Symposium (COSy) at Guelph. For
that I would like to thank the organizers J. Holbrook and D. Kribs,
and the Fields Institute for funding support to attend this event. I
would also like to thank the people in the Math and Stats department at the
University of Regina for their kind hospitality during my PIMS pdf
there, particularly to M. Argerami, D. Farenick and S. Fallat.

\section{Preliminaries}\label{spreliminares}

\noindent {\bf Some notations and terminology}. We denote by $\mat$ (resp. $\matsa$,
$\matpos$, $\U(n)$) the set of $n\times n$ complex (resp
selfadjoint, positive semi-definite, unitary) matrices, with
identity $1_n$. By a system of projections $\cP=\{ P_i\}_{i=1}^m$ in
$\mat$ we mean an ordered set of $n\times n$ complex orthogonal
projection matrices such that $\sum_{i=1}^mP_i=1_n$ (thus, the ranges
of $P_1,\ldots,P_m$ are pairwise orthogonal). Given a system of
projections $\cP=\{ P_i\}_{i=1}^m$ in $\mat$
we consider the compression $\C_\cP:\mat\rightarrow \mat$  induced
by $\cP$ given by $\C_\cP(S)=\sum_{i=1}^mP_iS\,P_i$.  Notice that
$\C_\cP$ is a trace preserving completely positive map. We shall
consider $\oplus_{i=1}^m \mathcal M_{d(i)}(\CC)\subseteq \mat$ as a
unital $*$-subalgebra of $\mat$.   If
$x\in \RR^n$ then we denote by $x^\downarrow\in \RR^n$ the vector
obtained from $x$ by rearranging the coordinates of $x$ in
non-increasing order. If $A\in \matsa$ then
$\lambda(A)=\lambda(A)^\downarrow\in \RR^n$ denotes the $n$-tuple of eigenvalues of $A$
counting multiplicities and arranged in non-increasing order.
If $S\in \mat$ then $\U_n(S)$, $\C_n(S)$ denote respectively the
unitary and contractive orbit of $S$ i.e. $\U_n(S)=\{U^*S\,U:\ U\in
\U(n)\}$, $\C_n(S)=\{V^*S\,V:\ V\in \mat, \ \|V\|\leq 1\}$. More
generally, $\U_n(\mathcal X)$, $\C_n(\mathcal X)$ denote the unitary
and contractive orbit of $\mathcal X\subseteq \mat$.
 We shall
denote the canonical basis of $\CC^n$ as $\{ e_i\}_{i=1}^n$. If
$\lambda\in \RR^n$ we denote by $\text{Diag}(\lambda)$ the diagonal
matrix with main diagonal $\lambda$. The set $\{1,\ldots,n\}$ is
denoted by $\langle n\rangle$. We denote by $\RR_{\geq 0}$ the set
of non-negative real numbers.

\subsection{Majorization in $\matsa$}

We begin by recalling the notion of vector majorization and
submajorization. If $x,\, y\in \RR^n$ then we say that $x$ is
submajorized by $y$, denoted $x\prec_w y$, if for $1\leq k\leq n$
then $\sum_{i=1}^k x^\downarrow_i\leq \sum_{i=1}^k y^\downarrow_i$.
If $x\prec _w y$ and moreover $\sum_{i=1}^n x_i=\sum_{i=1}^n y_i$,
we say that $x$ is majorized by $y$ and write $x\prec y$.
Vector majorization arises naturally in the theory of inequalities
between convex functionals. This notion is also related with the
so-called doubly-stochastic matrices. Finally, our main motivation
for the introduction of majorization is the fact that it describes
the relation between the spectrum and the main diagonal of an
hermitian matrix. Indeed we have
\begin{teo}[Schur-Horn]\label{prelis shvec}Let $x,\,y\in \RR^n$. Then, there exists an
hermitian (or real symmetric) matrix $A\in \matsa$ with main
diagonal $x$ and $\lambda(A)=y^\downarrow$ if and only if $x\prec y$.
\end{teo}

Ando extended in \cite{Ando} the notion of vector (sub)majorization
to that of (sub)majori\-zation between elements in $\matsa$ i.e. the
real vector space of hermitian matrices. Indeed, given $A,\,B\in
\matsa$ we say that $A$ is majorized (resp submajorized) by $B$,
denoted $A\prec B$ (resp $A\prec _w B$) if $\lambda(A)\prec
\lambda(B)$ (resp $\lambda(A)\prec_w \lambda(B)$).

Majorization between hermitian matrices (operators) is also related
with inequalities of convex functionals, doubly-stochastic maps and
the values of conditional expectations onto maximal abelian
subalgebras of $\mat$. In order to state the next result, in which
we summarize some well known facts, we introduce the following
terminology and notations. Recall that a doubly-stochastic map
$T:\mat\rightarrow \mat$ is a linear map such that $T(1)=1$
(unital), $T(C)\geq 0$ whenever $C\geq 0$ (positive) and such that
$\tr(T(X))=\tr(X)$ for every $X\in \mat$ (trace preserving). We define $\mathcal
E_\mathcal D:\mat\rightarrow \mat$ such that, for
$X=(x_{ij})_{ij}\in \mat$ then $\mathcal E_\mathcal
D(X)=\text{Diag}(x_{11},\ldots,x_{nn})$.  A
particularly important example of a doubly stochastic map is given
by $T(X)=\mathcal E_\cD(U^*XU)$ for a fixed unitary matrix $U$.
Notice that the Schur-Horn
theorem \ref{prelis shvec} can we re-stated as
\begin{equation}\label{rees sh} \{ \mathcal E_\mathcal
D(U^*\text{Diag}(y)\,U):\ U\in \U(n)\}=\{ \text{Diag}(x):x\in
\RR^n\, , \ x\prec y\}
\end{equation}
which is an spectral description of the set in the left-hand side of
the equality above. In what follows, we consider
$M_B(n)=\{\lambda\in \RR^n:\ \lambda\prec \lambda(B) \}\subseteq
\RR^n$.
\begin{teo}\label{prelimsh}
Let $A,\,B\in \matsa$. Then the following statements are equivalent:
\begin{enumerate}
\item\label{shc1} $\lambda(A)\in M_B(n)$ (or equivalently $A\prec B$).
\item For every convex function $f:\RR\rightarrow \RR$ we have
$\tr(f(A))\leq \tr(f(B))$.
\item There exists a doubly-stochastic map $T:\mat\rightarrow\mat$
such that $T(B)=A$.
\item\label{shc4} $\U_n(A)\cap \{ \mathcal E_\cD(U^* B\, U):\ U\in \U(n)\}\neq \emptyset$
or equivalently $$A\in \U_n( \{ \mathcal E_\cD(U^* B\, U):\ U\in
\U(n)\}).$$
\end{enumerate}
\end{teo}
We refer to the equivalence between \eqref{shc1} and \eqref{shc4} in
Theorem \ref{prelimsh} as the commutative (since $\mathcal
E_\cD(\mat)$ is a commutative unital $*$-subalgebra of $\mat$)
operator algebra version of the Schur-Horn theorem, which is an
spectral description of the relation in \eqref{shc4}.

There is a similar result for sub-majorization. But in order to get
a complete analogy with Theorem \ref{prelimsh}, we have to restrict
our attention to submajorization between positive semi-definite
matrices. Recall that a
doubly sub-stochastic map $T:\mat\rightarrow \mat$ is a positive
linear map such that $\tr(T(X))\leq \tr(X)$ for $X\in \mat$ (trace
reducing) and $T(1_n)\leq 1_n$ (sub-unital). A particularly
important example of a doubly sub-stochastic map is given by
$T(X)=\mathcal E_\cD(V^*XV)$ for a fixed contraction $V$. In what
follows, we consider $M^w_B(n)=\{\lambda\in (\RR_{\geq 0})^n:\
\lambda\prec_w \lambda(B) \}\subseteq \RR^n$ for $B\in \matpos$.

\begin{teo}\label{prelimcsh}
Let $A,\,B\in \matpos$. Then the following statements are
equivalent:
\begin{enumerate}
\item\label{cshc1} $\lambda(A)\in M_B^w(n)$ (or equivalently $A\prec_w B$).
\item For every convex non-decreasing function $f:[0,\infty)\rightarrow \RR$ we have
$\tr(f(A))\leq \tr(f(B))$.
\item There exists a doubly sub-stochastic map $T:\mat\rightarrow\mat$
such that $T(B)=A$.
\item\label{cshc4} $\U_n(A)\cap \{ \mathcal E_\cD(V^* B\, V):\ V\in \mat, \ \|V\|\leq 1\}\neq \emptyset$
or equivalently $$A\in \U_n( \{ \mathcal E_\cD(V^* B\, V):\ V\in
\mat, \ \|V\|\leq 1\}).$$
\end{enumerate}
\end{teo}
We refer to the equivalence between \eqref{cshc1} and \eqref{cshc4}
in Theorem \ref{prelimcsh} as the commutative contractive operator
algebra version of the Schur-Horn theorem, which is an spectral
description of the relation in \eqref{cshc4}.

\subsection{Klyachko's theory on sums of hermitian matrices}

We briefly describe some basic notions of Schubert varieties and
admissible $m$-tuples to state Theorem \ref{KlFF}. This result
summarizes the deep work of Klyachko \cite{Klya}, Friedland
\cite{fried} and Fulton \cite{ful}. For a detailed account on these
and related topics we refer the reader to \cite{ful2} and the
references therein.

 Let
$V_*=V_1\subset V_2\subset \cdots V_n=\CC^n$ be a complete flag on
$\CC^n$ i.e. $\dim(V_i)=i$ for $1\leq i\leq n$. Fix $1\leq r< n$ and
let $I=\{i_1,\ldots,i_r \}\subset\langle n\rangle$ with $1\leq
i_1<i_2<\cdots<i_r\leq n$. Denote
$$I'=\{i'_1,\ldots,i'_r\}, \ \ \ i'_j=n+1-i_{r+1-j}, \ \ j=1,\ldots,r.$$
Let $X=Gr(r,\CC^n)$ be the Grassmann variety of all $r$-dimensional
subspaces $L$ of $\CC^n$. Let $\Omega_I(V_*)$ be the \emph{Schubert
variety} in $X$ defined by $$ \Omega_I(V_*):=\{ L\in X:\ \dim(L\cap
V_{i_j})\geq j \ \text{ for } \ 1\leq j\leq r\}.$$ An $(m+1)$-tuple
$(I_0,\ldots,I_m)$ of subsets $I_0,\ldots,I_m$ of $\langle n\rangle$,
each of cardinality $r$ ($1\leq r<n$) is called \emph{admissible},
if for any $(m+1)$ complete flags $V_*^0,\ldots,V_*^m$ of $\CC^n$
the following condition holds: $$ \Omega_{I_0}(V_*^0)\cap \left(
\bigcap _{j=1}^m \Omega_{I_j '} (V_*^j)\right)\neq \emptyset.  $$ We
will use the following notations. Let $|J|$ denote the cardinal of
the set $J$ and let
$$x[I]:=\sum_{i\in I}x_i, \ \ x\in \RR^n, \ \ I\subseteq\langle n
\rangle,\ \ |I|\geq 1.$$
\begin{teo}[\cite{Klya,fried,ful}]\label{KlFF}
Let $(\lambda^i)_{i=0}^m\in(\RR^n)^{(m+1)}$ be an $(m+1)$-tuple of
vectors in $\RR^n$. Then
\begin{enumerate}
\item\label{itkly} There exist $m+1$ matrices $A_0,\ldots,A_m\in \matsa$
such that $\lambda(A_i)=\lambda^i$ for $0\leq i\leq m$ and
$A_0=\sum_{i=1}^m A_i$ if and only if $\lambda^0[\langle
n\rangle]=\sum_{i=1}^m\lambda^i[\langle n\rangle]$ and
\begin{equation}\label{comp K}
\lambda^0[I_0']\geq \sum_{j=1}^m\lambda^j[I_j'], \ \ \text{ for
every admissible} \ (m+1)\text{-tuple}\ (I_j)_{j=0}^m.\end{equation}
\item\label{iFF} There exist $m+1$ matrices $A_0,\ldots,A_m\in \matsa$
such that $\lambda(A_i)=\lambda^i$ for $0\leq i\leq m$ and $A_0\geq
\sum_{i=1}^m A_i$ if and only if $\lambda^0[\langle
n\rangle]\geq\sum_{i=1}^m\lambda^i[\langle n\rangle]$ and the
inequalities \eqref{comp K} hold.
\end{enumerate}
\end{teo}

We point out that the inequalities in \eqref{comp K} are rather the
dual inequalities to those that appear in \cite{fried,ful,Klya}. The
fact that the theorem above follows from those papers is a
consequence of the following equalities: for
$I=\{i_1,\ldots,i_r\}\subseteq \langle n\rangle$ as above and $\lambda \in \RR^n$ such that $\lambda = \lambda ^\downarrow$ then
$$(-\lambda)^\downarrow [I] =\sum_{j=1}^r -\lambda_{n+1-\,i_j}=-\sum_{j=1}^r \lambda_{n+1-\,i_{r+1-j}}=- (\lambda[I']).$$
As noted in \cite{fried}, \eqref{itkly} follows from \eqref{iFF}.
 The inequalities in \eqref{comp
K} are referred to as Klyachko's compatibility inequalities. We say
that an $(m+1)$-tuple $(\lambda^i)_{i=0}^m\in(\RR^n)^{(m+1)}$
satisfies Klyachko's compatibility inequalities if it satisfies the
family of inequalities given in \eqref{comp K}. Note that these
inequalities depend on the admissible $(m+1)$-tuples of $\langle
n\rangle$.

\section{Non commutative Schur-Horn theorems}

\subsection{NC Schur-Horn theorems for block diagonal compressions}

 We say that $\{ P_i\}_{i=1}^m\subseteq \mat$ is a
system of \emph{coordinate} projections if there exists a partition
$\{\mathcal J_i\}_{i=1}^m$ of $\langle n \rangle$ by increasing
subintervals (i.e. if $k_1\leq k\leq k_2$ with $k_1,\,k_2\in \mathcal J_i$ then $k\in  \mathcal J_i$ for $1\leq i\leq m$ and 
 if $k\in \mathcal J_i$, $l\in \mathcal J_j$ then $k\leq
l$ whenever $i\leq j$) such that $P_i$ is the projection onto
span$\{e_k,\ k\in \mathcal J_i\}$ for $1\leq i\leq m$. Notice that
in this case $\C_\cP:\mat\rightarrow \oplus_{i=1}^m \mathcal
M_{d(i)}(\CC)\subseteq \mat$, where rank$\,(P_i)=d(i)$ for $1\leq
i\leq m$. If $\mathcal Q=\{Q_i\}_{i=1}^m\subseteq \mat$
is an arbitrary system of projections with rank$\,(Q_i)=d(i)$ for
$1\leq i\leq m$ then there exists a unitary operator $W\in \U(n)$
such that $Q_i=W^*P_iW$ for $1\leq i\leq m$ and hence $\mathcal
C_\mathcal Q(X)=W^*\,\mathcal C_\cP(W\, X\, W^*)W$ for $X\in \mat$.
Hence, these coordinate systems of projections are a model for more
general systems of projections.

The following result can be considered as non-commutative
contractive Schur-Horn theorem for positive semi-definite matrices
with respect to block diagonal compressions.

\begin{teo}\label{teo contrac sh} Let $\cP=\{P_i\}_{i=1}^m\subseteq \mat$ be a
system of coordinate projections with rank$(P_i)=d(i)$ and let
$\C_\cP:\mat\rightarrow \oplus_{i=1}^m \mathcal M_{d(i)}(\CC)$ be
the compression induced by $\cP$. If $S\in \matpos$ and $S_i\in
\mathcal M_{d(i)}(\CC)^+$ are such that
$\lambda(S_i)=\lambda^i\in(\RR_{\geq 0})^{d(i)}$ for $1\leq i\leq
m$, then the following statements are equivalent:
\begin{enumerate}
\item\label{1} There exists a contraction $V\in \mat$ such that
$$\C_\cP(V^*SV)=\oplus_{i=1}^m S_i.$$
\item \label{2} There exist unitary matrices $V_i\in \U(n)$
for $1\leq i\leq m$ such that
$$S\geq\sum_{i=1}^m V_i ^* (\oplus_{j=1}^m \delta_{ij}\, S_j)\,
V_i$$ where $\delta_{ij}$ is Kronecker's delta function.
\item \label{3} There exist a contraction $W\in \mat$ and
 unitary matrices $V_i\in \U(n)$ for
$1\leq i\leq m$ such that $$W^*SW=\sum_{i=1}^m V_i ^*
(\oplus_{j=1}^m \delta_{ij}\, S_j)\, V_i.$$
\item \label{4} The $(m+1)$-tuple $(\lambda(S),
(\lambda^1,0_{n-d(1)}),\ldots, (\lambda^m,0_{n-d(m)}))$ satisfies
Klyachko's compatibility inequalities plus $\tr(S)\geq \sum_{i=1}^m
\tr(S_i)$.
\end{enumerate}
\end{teo}
\begin{proof}
The equivalence between \eqref{2} and \eqref{3} is well known, while
the equivalence of \eqref{2} and \eqref{4} is item \eqref{iFF} in
Theorem \ref{KlFF} (see \cite{ful}, \cite{fried}).

The equivalence of \eqref{1} and \eqref{2} is essentially the same
as the proof of \cite[Thm 2.2]{LiPoon}, so we sketch it. Assume that
\eqref{1} holds for some contraction $V\in \mat$. We define the
matrices $T_i=S^{1/2}V\,P_i\in \mat$ for $1\leq i\leq m$. Then it is
straightforward that $T_i^*T_i=\oplus_{j=1}^m \delta_{ij}\, S_j$
while
 $\sum_{i=1}^m T_iT_i^*\leq S$. If $V_i$ are unitary matrices
such that $T_iT_i^*=V_i^* (T_i^*\,T_i)V_i$ then \eqref{2} holds for
these unitary matrices.
For the converse, assume now that \eqref{2} holds. We consider
$$R= \sum_{i=1}^m V^*_i \, (\oplus_{j=1}^m
\delta_{ij}\, S_j )^{1/2}=\sum_{i=1}^m V^*_i\, ( \oplus_{j=1}^m
\delta_{ij}\, S_j^{1/2}) P_i.$$ Then, we have that $0\leq RR^*\leq
S$ and therefore there exists a contraction $W$ such that
$RR^*=W^*SW$. On the other hand,
\begin{eqnarray}\label{c de rr2}\C_\cP(R^*R)&=& \C_\cP(( \sum_{k=1}^m
P_k\,(\oplus_{j=1}^m \delta_{kj}\, S_j^{1/2})U_k ) (\sum_{i=1}^m
U^*_i\, ( \oplus_{j=1}^m \delta_{ij}\, S_j^{1/2}) \, P_i )) \\
\nonumber &=&\oplus_{i=1}^m  S_i. \end{eqnarray}
If $U\in \mat$ is a unitary matrix such that
$R^*R=U^*(R\,R^*)U=U^*(W^*SW)U$ then \eqref{1} holds for the
contraction $V=WU$.
\end{proof}

In what follows we denote by $\mathbf e^d\in \RR^d$ the vector with
all coordinates equal to 1. Next we derive \cite[Thm 2.2]{LiPoon} from Theorem \ref{teo contrac sh}, which
is a non-commutative
versions of the Schur-Horn theorem for hermitian matrices with
respect to block diagonal compressions. (Compare this result with
\eqref{rees sh}).

\begin{teo}[\cite{LiPoon}]\label{teo KSH}
Let $\cP=\{P_i\}_{i=1}^m\subseteq \mat$ be a system of coordinate
projections in $\mat$ with rank$\,(P_i)=d(i)$ and let
$\C_\cP:\mat\rightarrow \oplus_{i=1}^m \mathcal M_{d(i)}(\CC)$ be
the compression induced by $\cP$. Let $S\in \matsa$ and $S_i\in
\mathcal M_{d(i)}(\CC)^{sa}$ be such that
$\lambda(S_i)=\lambda^i\in\RR^{d(i)}$ for $1\leq i\leq m$ and let $\alpha\in \RR$ be such that $S+\alpha\,1_n\in\matpos$. Then the
following statements are equivalent:
\begin{enumerate}
\item\label{eshh} There exists a unitary matrix $U\in \U(n)$ such that
$$\C_\cP(U^*SU)=\oplus_{i=1}^m S_i.$$
\item There
exist unitary matrices $U_i\in \U(n)$ for $1\leq i\leq m$ such that
$$S+\alpha\,1_n=\sum_{i=1}^m U^*_i(\oplus_{j=1}^m \delta_{ij}\
\left(S_i+\alpha\, 1_{d(i)}) \right)\,U_i.$$
\item The $(m+1)$-tuple
$$(\lambda(S)+\alpha\, \mathbf e^n,(\lambda^1+ \alpha\, \mathbf
e^{d_1} ,0_{n-d(1)}),\ldots, (\lambda^m + \alpha\, \mathbf e^{d_m},
0_{n-d(m)}))\in (\RR_{\geq 0}^n)^{(m+1)}$$ and it satisfies
Klyachko's compatibility inequalities plus $\tr(S)=\sum_{i=1}^m
\tr(S_i)$.
\end{enumerate}
\end{teo}

\begin{proof}
Note that, for any unitary matrix $U\in \U(n)$ and $A\in \mat$ such that $\cC_\cP(A)=\oplus_{i=1}^m A_i$ we have \begin{equation*}
\C_\cP(U^*(A+\alpha\,1_n)U)=\C_\cP(U^*A\,U)+\alpha\,
1_n=\oplus_{i=1}^m A_i+\alpha\, 1_n=\oplus_{i=1}^m (A_i+\alpha\,
1_{d(i)})
\end{equation*} From this it is easy to see that we may assume $S\in \matpos$ and $\alpha=0$. In this last case, the result follows from Theorem \ref{teo contrac sh} and the fact that, for $A,\,B\in \matsa$ such that $A\leq B$ and $\tr(A)=\tr(B)$ we have that $A=B$.
\end{proof}

\begin{rem} Using Theorem \ref{teo KSH} ($\alpha=0$) and the classical Schur-Horn Theorem, we can see that if $\mathbf a=(a_1,\ldots,a_n),\,\mathbf b=(b_1,\ldots,b_n)\in
(\RR_{\geq 0})^n$ then the $(n+1)$-tuple $(\mathbf b, a_1 \cdot
e_1,\ldots, a_n\cdot e_1)$ satisfies Klyachko's compatibility
inequalities together with $\sum_{i=1}^n a_i=\sum_{i=1}^n \lambda_i$
if and only if
$\mathbf a\prec \mathbf b$.

This last fact suggests that there might be alternative sets
of linear inequalities for the spectral conditions in Theorem \ref{teo KSH}, which are less complex than Klyachko's compatibility inequalities. Such a reduction of the complexity of this problem has been done by C.K. Li and T.Y. Poon in \cite[Thm 3.3]{LiPoon} They find a reduced set of the set of Klyachko´s inequalities to be checked in order that the $(m+1)$-tuple $(\lambda,(\lambda_1,0_{n-d(1)}),\ldots(\lambda_m,0_{n-d(m)}))$ satisfies all of Klyachko´s inequalities. They show that the complexity of this reduced set actually depends on the dimensions $d(1),\ldots,d(m)$.
\end{rem}

One of the most important consequences of the Schur-Horn theorem as
stated in \eqref{rees sh}, is the fact that the left-hand side of
that equality is a convex set (because the right-hand side is easily
seen to be convex). As the following proposition shows this is a
particular feature of the diagonal compression $\mathcal E_\mathcal
D$ onto a maximal abelian $*$-subalgebra of $\mat$.

\begin{pro}\label{no conv}
Let $\cP=\{P_i\}_{i=1}^m\subseteq \mat$ be a system of coordinate
projections in $\mat$ with rank$\,(P_i)=d(i)$ and let
$\C_\cP:\mat\rightarrow \oplus_{i=1}^m \mathcal M_{d(i)}(\CC)$ be
the compression induced by $\cP$. The following statements are
equivalent:
\begin{enumerate}
\item The set $\C_\cP(\U_n(S))$ is convex for every $S\in \matpos$.
\item For every $1\leq i\leq m$ we have $d_i=1$ (and hence $m=n$).
\end{enumerate}
\end{pro}

\begin{proof} Let $\cP=\{P_i\}_{i=1}^m$ be as above
and assume first that $d(i)=1$ for every $1\leq i\leq m$. Hence
$m=n$ and the convexity of $\C_\cP(\U_n(S))$ follows from the
classical Schur-Horn theorem (see \eqref{rees sh}).

For the converse we assume, without loss of generality, that
$d(1)\geq 2$. We define $$ S=\begin{pmatrix} 2 & 0 \\ 0 &
4\end{pmatrix}\oplus 0_{(n-2)} \ \ \text{ and } \ \
V=\begin{pmatrix} 0 & 1 \\ 1 & 0\end{pmatrix}\oplus 1_{(n-2)}.$$ In
this case
\begin{equation}T:=\frac{1}{2}(\C_\cP(S)+\C_\cP(V^*
SV))=\frac{1}{2}(S+V^* SV) =\begin{pmatrix} 3 & 0\\ 0 &
3\end{pmatrix} \oplus 0_{(n-2)} .\end{equation} Assume now that
there exists $ U \in \U(n)$ such that $
\C_\cP(U^*SU)=T$. But, since $U^*SU\geq 0$ and $d(1)\geq 2$, the
equality above implies that $U^*S \, U=T$. This last fact is a
contradiction, since these two matrices have different spectrum.
\end{proof}


\begin{rem}\label{sist gen de proj}
Let $\mathcal Q=\{ Q_i\}_{i=1}^m\subseteq \mat $ be an arbitrary
system of projections with rank$\,(Q_i)=d(i)$ for $1\leq i\leq m$.
Then, as remarked at the beginning of this section, there exists a
coordinate system of projection $\{P_i\}_{i=1}^m$ and a unitary
operator $W\in \U(n)$ such that $W^*P_iW=Q_i$ for $1\leq i\leq m$
and hence $\mathcal C_\mathcal Q(X)=W^*\,\mathcal C_\cP(W\, X\,
W^*)W$ for $X\in \mat$. Using these facts it is easy to see that the
results of this section extend to results about the general system
$\mathcal Q$, but based on $\mathcal P$ and $W$. Still, we point out
that in general there is no canonical choice for $W$ given $\mathcal P$ and
$\mathcal Q$ as above.
\end{rem}

\subsection{A NC Schur-Horn theorem for partial traces}\label{lasec22}

Partial traces were brought to attention of the linear algebra
community by \cite{bhat}, although here we present this notion in a
rather different way. Recall that given $\mathcal
M_{d}(\CC)\otimes\mathcal M_{m}(\CC)$ there are two natural partial
traces associated, $\Tr_m:\mathcal M_{d}(\CC)\otimes\mathcal
M_{m}(\CC)\rightarrow \mathcal M_d(\CC)$ and $\Tr_d:\mathcal
M_{d}(\CC)\otimes\mathcal M_{m}(\CC)\rightarrow \mathcal M_m(\CC)$,
determined by the following properties: for every $A\in \mathcal
M_d(\CC)$, every $B\in\mathcal M_m(\CC)$  and every $C\in \mathcal
M_{d}(\CC)\otimes \mathcal M_m(\CC)$ then
\begin{equation}\label{defi de part trac}
\tr(\Tr_m(C)\, A)=\tr(C\,(A\otimes 1_m))\,, \ \ \ \tr(\Tr_d(C)
\,B)=\tr(C\,(1_d\otimes B)).
\end{equation}
Notice that the traces to the left and right of equality signs above
are not the same, but we will allow this abuse of notation. Indeed,
the trace in $\mathcal M_{d}(\CC)\otimes\mathcal M_{m}(\CC)$ is
defined on elementary tensors as $\tr(A\otimes B)=\tr(A)\cdot
\tr(B)$.

Let us now identify $\mathcal M_{d}(\CC)\otimes\mathcal M_{m}(\CC)$
 with $\mathcal M_{d\cdot m}(\CC)$ by
 \begin{equation}\label{la ident}
 A\otimes B\approx (b_{ij}\,A)_{i,j=1}^m
 \end{equation}
By means of this identification, if $C=(C_{ij})_{ij=1}^m\in \mathcal
M_{d\cdot m}(\CC)$ with $C_{ij}\in \mathcal M_d(\CC)$ for $1\leq
i,\,j\leq m$, we can see that the partial traces become
\begin{equation}\label{defi alter trac}
\Tr_m(C)=\sum_{i=1}^m C_{ii}\in \mathcal M_d(\CC) \,, \ \ \
\Tr_d(C)=( \tr(C_{ij}))_{i,j=1}^m\in \mathcal M_m(\CC),
\end{equation} since these definitions satisfy the conditions in \eqref{defi de part trac} using \eqref{la ident}.
Notice that there is a symmetric situation for $\mathcal M_d(\CC)$
and $\mathcal M_m(\CC)$ with respect to the algebra $\mathcal
M_d(\CC)\otimes \mathcal M_m(\CC)$. The fact that the expressions in
\eqref{defi alter trac} are not symmetric is a consequence of our
particular identification \eqref{la ident}. In what follows, given $\mathcal X\subseteq \mathcal M_{d\cdot m}(\CC)$ we denote by $\Tr_m(\mathcal X)$ the set of all values $\Tr_m(x)$ with $x\in \mathcal X$.

\begin{teo}\label{teo par trac} Let us identify
$\mathcal M_d(\CC)\otimes \mathcal M_m(\CC)$ with $\mathcal
M_{d\cdot m}(\CC)$ as before, so that we get the previous
description of $\Tr_m$.
\begin{enumerate}
\item \label{it1} If $S\in \mathcal M_{d\cdot m}(\CC)^{sa}$ then there exists
$D_S(d,m)\subset \RR^d$ such that
$$ \Tr_m(\U_{d\cdot m}(S))=\{A\in \mathcal M_d(\CC)^{sa}:\ \lambda(A)\in D_S(d,m) \}.$$
\item  \label{it2} If $S\in \mathcal M_{d\cdot m}(\CC)^{+}$ then there exists
$D^w_S(d,m)\subset (\RR_{\geq 0})^d$ such that
$$ \Tr_m(\cC_{d\cdot m}(S))=\{A\in \mathcal M_d(\CC)^{+}:\ \lambda(A)\in D_S^w(d,m) \}.$$
\end{enumerate}
\end{teo}

\begin{proof}
To prove \eqref{it1} we first define $ A_S(d,m)\subset(\RR^d)^m$ as
the set containing all $(\lambda^i)_{i=1}^m$ where
$\lambda^i=(\lambda^i)^\downarrow\in \RR^d$ for $1\leq i\leq m$, and
such that the $(m+1)$-tuple
$$(\lambda(S)+\|S\|\, \mathbf e^n,(\lambda^1+ \|S\|\, \mathbf
e^{d} ,0_{d(m-1)}),\ldots, (\lambda^m + \|S\|\, \mathbf e^{d},
0_{d(m-1)}))\in (\RR_{\geq 0}^{(d\cdot m)})^{(m+1)}$$ and it
satisfies Klyachko's compatibility inequalities plus the equality
$\tr(S)=\sum_{i=1}^m \sum_{j=1}^{d}\lambda^i_j$. We then define
$D_S(d,m)$ as the set containing all vectors
$\lambda=\lambda^\downarrow\in \RR^d$ such that there exists
$(\lambda^i)_{i=1}^m\in A_S(d,m)$ so that the $(m+1)$-tuple
$(\lambda,\lambda^1,\ldots,\lambda^m)$ satisfies Klyachko's
compatibility inequalities plus the condition
$\sum_{i=1}^d\lambda_i=\tr(S)$. The fact that $D_S(d,m)$ as defined
above has the desired properties is a consequence of Theorems
\ref{KlFF} and \ref{teo KSH}.

The proof \eqref{it2} is analogous. We define  first $ A_S^w(d,m)$
as the set containing all $(\lambda^i)_{i=1}^m$ where $\lambda^i\in
\RR_{\geq 0}^d$ for $1\leq i\leq m$, and  such that the
$(m+1)$-tuple
$$(\lambda(S), (\lambda^1,0_{d(m-1)}),\ldots, (\lambda^m,
0_{d(m-1)}))$$ satisfies Klyachko's compatibility inequalities plus
$\tr(S)\geq \sum_{i=1}^m \sum_{j=1}^{d}\lambda^i_j$. We then define
$D_S^w(d,m)$ as the set containing all vectors
$\lambda=\lambda^\downarrow\in (\RR_{\geq 0})^d$ such that there
exists $(\lambda^i)_{i=1}^m\in A_S^w(d,m)$ such that the
$(m+1)$-tuple $(\lambda,\lambda^1,\ldots,\lambda^m)$ satisfies
Klyachko's compatibility inequalities plus the condition
$\sum_{i=1}^d\lambda_i= \sum_{i=1}^ m\sum_{j=1}^d\lambda^i_j$. The
fact that $D_S^w(d,m)$ as defined above has the desired properties
is now a consequence of Theorems \ref{KlFF} and \ref{teo contrac
sh}.
\end{proof}

\subsection{The non-commutative Schur-Horn theorems}

We  begin by recalling some basic facts about unital $*$-subalgebras
and trace preserving conditional expectations in $\mathcal
M_n(\CC)$. Let $\mathcal A\subseteq \mat$ be a unital
$*$-subalgebra. Then, $\mathcal A$ is a subspace of the finite
dimensional complex inner product space $(\mat,\langle\cdot\,,\cdot
\rangle_{\tr})$ where $\langle A, B\rangle_{\tr}=\tr(B^*A)$. Thus,
we can consider $\mathcal E_\mathcal A$ the orthogonal projection
with respect to $\langle\cdot\,,\cdot \rangle_{\tr}$ onto $\mathcal
A$. That is, $\mathcal E_\mathcal A:\mat\rightarrow \mat$ is a
linear, $\mathcal E_\mathcal A \circ \mathcal E_\mathcal A=\mathcal
E_\mathcal A$ and
\begin{equation}\label{defi esp cond}
 \tr(C^* \mathcal E_\mathcal A(B))=\tr(\mathcal E_\mathcal
A(C)^* B )\, ,\ \ \ \ \mathcal E_\mathcal A(A)=A \, , \ \ \forall
A\in \mathcal A.\end{equation}
 In the operator algebra context $\mathcal
E_\cA$ is called the trace preserving conditional expectation (TCE)
onto $\mathcal A$; the fact that it is trace preserving is a
consequence of the relations in \eqref{defi esp cond} setting $C=1$
and recalling that $1\in \mathcal A$. The TCE is uniquely determined
by the previous properties.

We consider first the following two examples. Let $\cP=\{
P_i\}_{i=1}^m$ be a system of coordinate projections with
rank$\,(P_i)=d(i)$ for $1\leq i\leq m$, and consider $\mathcal
A=\oplus_{i=1}^m \mathcal M_{d(i)}(\CC)\subset \mathcal M_n(\CC)$.
Then $\mathcal A$ is a unital $*$-subalgebra of $\mathcal M_n(\CC)$
and the compression $\C_\cP=\mathcal E_\mathcal A$ is the TCE onto
$\mathcal A$. In a similar way, if we now consider the
identification of $\mathcal M_d(\CC)\otimes \mathcal M_m(\CC)$ with
$\mathcal M_{d\cdot m}(\CC)$ described at the beginning of section
\ref{lasec22} then the algebra $\mathcal M_d(\CC)\otimes 1_m$
regarded inside of $\mathcal M_{d\cdot m}(\CC)$ is a unital
$*$-subalgebra of $\mathcal M_{d\cdot m}(\CC)$ (the algebra of
$m\times m$ block diagonal matrices with constant diagonal blocks).
In this case, we can describe the TCE onto $\mathcal A$
 in terms of the partial trace $\Tr_m$ by $\mathcal
E_\mathcal A(C)=\frac{1}{m}\Tr_m(C)\otimes 1_m$.

In general, a unital $*$-subalgebra of $\mathcal M_n(\CC)$  can be
described, up to conjugation by a unitary matrix $U\in \mat$, as a
direct sum of $m$ blocks, each of the form $M_{d(i)}\otimes
1_{c(i)}$ for $1\leq i\leq m$ and such that $\sum_{i=1}^m d(i)\
c(i)=n$. The list $(d(i),c(i))_{i=1}^m$, that we call the spectral
list, is invariant under unitary conjugations. Moreover, two unital
$*$-subalgebras $\mathcal A,\,\mathcal B\subset \mat$ with spectral
lists $(d_\mathcal A(i),c_\mathcal A(i))_{i=1}^m$  and $(d_\mathcal
B(i),c_\mathcal B(i))_{i=1}^r$ are unitary conjugate (i.e. there
exists a unitary $U\in \U(n)$ with $U^*\mathcal A U=\mathcal B$) if
and only if $m=r$ and there exists a permutation $\sigma\in \mathbb
S_m$ such that $ (d_\mathcal A(i),c_\mathcal A(i))=(d_\mathcal
B(\sigma(i)),c_\mathcal B(\sigma(i)))$ for $1\leq i\leq m$. In this
case we say that the lists $(d_\mathcal A(i),c_\mathcal
A(i))_{i=1}^m$  and $(d_\mathcal B(i),c_\mathcal B(i))_{i=1}^r$ are
\emph{equivalent}. Strictly speaking, the spectral list of a unital $*$-subalgebra is defined only up to equivalence, but we shall allow this abuse of language as it will not cause any problems with the notions to be considered.

 If the spectral list of a unital $*$-subalgebra
$\mathcal A$ is \emph{multiplicity free} i.e. it is of the form
$(d(i),1)_{i=1}^m$ then we say that $\mathcal A$ is
\emph{multiplicity free}. The multiplicity free algebras (lists) are
in some sense the well-behaved algebras (lists) in our context.

Let $\mathcal A=\oplus _{i=1}^m \mathcal M_{d(i)}\otimes 1_{c(i)}$
be a unital $*$-subalgebra of $\mat$ with spectral list
$(d(i),c(i))_{i=1}^m$ and let $\cP=\{P_i \}_{i=1}^m$ be a system of
coordinate projections with rank$\,(P_i)=d(i)\cdot c(i)$. Then, the
TCE  onto $\mathcal A$ can be described in
terms of block diagonal compressions and partial traces as
\begin{equation}\label{la esp}
 \mathcal E_\mathcal A(B)=\oplus_{i=1}^m
 \frac{1}{c(i)} \Tr_{c(i)}(B_i)\otimes 1_{c(i)}, \ \ \text{ where } \ \ \C_\cP(B)=\oplus_{i=1}^m\,
 B_i.
 \end{equation}
If $\mathcal B$ is a unital $*$-subalgebra  with an equivalent
spectral list to that of $\mathcal A$ then, as stated before, there
exists a unitary $U\in \U(n)$ such that $U^*\mathcal AU=\mathcal B$,
so
\begin{equation}\label{relac entre esp}\mathcal E_\mathcal B(C)=
U^*\mathcal E_\mathcal A(U\,C\,U^*)U\end{equation} This last fact
can be verified using the uniqueness of the TCE onto $\mathcal B$.
In what follows, given $\mathcal A,\,\mathcal X\subseteq \mat$ with $\mathcal
A$ a unital $*$-subalgebra and $\mathcal X$ an arbitrary set, we denote by
 $\mathcal E_\mathcal A(\mathcal X)$ the set of all values $\mathcal
E_\mathcal A(x)$ for $x\in \mathcal X$. The following result is an immediate
consequence of \eqref{relac entre esp}.
\begin{lem}\label{buena defi}
Let $\mathcal A,\,\mathcal B$ be $*$-subalgebras of $\mat$ with
equivalent spectral lists. Then, $$ \U_n(\mathcal E_ \mathcal
A(\U_n(S)))=\U_n(\mathcal E_\mathcal B(\U_n(S))) \ \ \text{ and } \
\ \U_n(\mathcal E_ \mathcal A(\C_n(S)))=\U_n(\mathcal E_\mathcal
B(\C_n(S))).$$
\end{lem}
We now prove the (finite dimensional operator algebra version of the) NC-Schur-Horn theorem in the Introduction. Notice that similar considerations
to those in Remark \ref{sist gen de proj} also apply to this
context. Recall that if $\lambda^i\in \RR^{d(i)}$ for $1\leq i\leq
m$ with $\sum_{i=1}^m d(i)=n$ then we denote by
$[\lambda^i]_{i=1}^m\in \RR^n$ the vector obtained by juxtaposition
of the vectors $\lambda^ i$'s i.e $\lambda=(\lambda^1_1,\ldots,
\lambda^1_{d(1)}, \lambda^2_1,\ldots, \lambda^m_{d(m)}).$

\begin{proof}[Proof of the NC-Schur-Horn theorem ]  Let us define $c=\sum_{i=1}^m c(i)\in \NN$ and let
$\mathbf k =(k(i))_{i=1}^{c}$ be the list given by $$k(\sum_{r=1}^
{i-1} c(r)+j)=d(i)\ \text{ for }\  1\leq i\leq m\, , \ \ 1 \leq
j\leq c(i).$$
 We define first $D_B(\mathcal A)$ as the set containing
  all $c$-tuples $(\mu^i)_{i=1}^c$ with
   $\mu^i\in \RR^{k(i)}$, $\mu^i=(\mu^i)^\downarrow$ for $1\leq i\leq c$ and such that the
   $(c+1)$-tuple
 $$(\lambda(B)+\|B\|\cdot\mathbf e^{n},
 (\mu^1+\|B\|\cdot\mathbf e^{k(1)},0_{n-k(1)}),
 \ldots, (\mu^c+\|B\|\cdot \mathbf e^{k(c)},0_{n-k(c)}))\in (\RR^n_{\geq 0})^c$$
and it satisfies Klyachko's compatibility inequalities plus
$\tr(B)=\sum_{i=1}^c\sum_{j=1}^{k(i)}\mu^i_j$. If we let
$\cP=\{P_i\}_{i=1} ^c$ be the system of coordinate projections with
rank$\,(P_i)=k(i)$ for $1\leq i\leq c$ then, by Theorem \ref{teo
KSH}, $(\mu^i)_{i=1}^c\in D_B(\mathcal A)$ if and only if it
can be realized as $\mu^i=\lambda(S_i)$ for $1\leq i\leq c$,
where $\C_\cP(F^*BF)=\oplus_{i=1}^c S_i$ for some $F\in \U(n)$.

We now define $N_B(\mathcal A)$ as the set containing all
$\lambda=(\lambda^i)_{i=1}^c$, where
$\lambda^i=(\lambda^i)^\downarrow\in \RR^{k(i)}$ for $1\leq i\leq c$
for which there exists $(\mu^i)_{i=1}^c\in D_B(\mathcal A)$  such
that, if $t(i)=\sum_{j=1}^{i-1}c(j)+1$ for $1\leq i\leq m$  then
 \begin{itemize}
\item[a)] $\lambda^{t(i)}=\lambda^t$ for every
$t(i)\leq t\leq t(i+1)-1$.
\item[b)] For $1\leq i\leq m$ the $(c(i)+1)$-tuples (note that $k(t(i))=d(i)$)
 \begin{equation}\label{la vs mu}
 (c(i)\,\lambda^{t(i)},\mu^{t(i)},\ldots, \mu^{t(i+1)-1})\in
 (\RR^{d(i)})^{c(i)+1}\end{equation}
satisfy Klyachko's compatibility
inequalities plus the condition
\begin{equation}\label{la igual mu}c(i)\,\sum_{j=1}^{d(i)}\lambda^{t(i)}_j=
\sum_{j=t(i)}^{t(i+1)-1}\sum_{r=1}^{d(i)}\mu^{j}_r.\end{equation}
 \end{itemize}Finally, we define $M_B(\mathcal A)$ as  the set containing all
vectors $\eta=[\eta^i]_{i=1}^m$ where $\eta^i\in \RR^{c(i)\,d(i)}$
for $1\leq i\leq m$ and such that there exists
$\lambda=(\lambda^j)_{j=1}^c\in N_B(\mathcal A)$ with
$\eta^i=[\lambda^{t(i)},\ldots,\lambda^{t(i+1)-1}]^\downarrow$ for
$1\leq i\leq m$.

Now we show that if $A=\oplus_{i=1}^m A_i\otimes 1_{c(i)}\in
\mathcal A$  is such that $[\lambda(A_i\otimes 1_{c(i)})]_{i=1}^m\in
M_B(\mathcal A)$ then there exists a unitary matrix $U\in \mat$ such
that $A=\mathcal E_\mathcal A(U^*BU)$. Recall that in this case the
TCE onto $\mathcal A=\oplus_{i=1}^m \mathcal M_{d(i)}(\CC)\otimes
1_{c(i)}$ is given by
\begin{equation}\label{el a} \mathcal E_\mathcal A(X)=
\oplus_{i=1}^m \frac{1}{c(i)} \sum_{j=t(i)}^{t(i+1)-1} X_{j}\otimes
1_{c(i)} \, , \ \ \text{ with } \ \ \C_\cP(X)=\oplus_{i=1}^c\,
X_i\end{equation} where $\cP=\{P_i \}_{i=1}^c$ is as before. For
$1\leq i\leq m$ and $t(i)\leq j\leq t(i+1)-1$  let us define
$\lambda^j:=\lambda(A_i)$ and let $\lambda:=(\lambda ^j)_{j=1}^c$.
By hypothesis there exists $\mu=(\mu^i)_{i=1}^c\in D_B(\mathcal A)$ such that \eqref{la vs mu} and \eqref{la igual mu}
hold for $\lambda$ and $\mu$. As remarked before, in this case there exists a unitary
$F\in\U(n)$ such that $\C_\cP(F^*BF)=\sum_{i=1}^c S_i$ and
$\lambda(S_i)=\mu^i$ for $1\leq i\leq c$.
By condition b) and Theorem \ref{KlFF}, for
$1\leq i\leq m$ there exist unitaries $W_{i,t(i)},\ldots,
W_{i,t(i+1)-1}\in \U(d(i))$ such that
\begin{equation} \label{exis los uni}
c(i)\, A_i=\sum_{j=t(i)}^{t(i+1)-1}
W^*_{i,j}\ S_i\  W_{i,j}
\end{equation}
If we now define
$W=\oplus_{i=1}^m\oplus_{j=t(i)}^{t(i+1)-1}W_{i,j}\in \U(n)$ then,
by \eqref{exis los uni} we have
\begin{equation}
\label{otra
eq}\C_\cP(W^*F^*B\,F\,W)=W^*\,\C_\cP(F^*BF)W=\oplus_{i=1}^m
\oplus_{j=t(i)}^{t(i+1)-1}  W^*_{i,j}\ S_i \ W_{i,j}
\end{equation}
and hence, using a) above, \eqref{el a} and \eqref{otra eq} we get
$$ \mathcal E_\mathcal A(W^*F^*\,B\,F\,W)=
\oplus_{i=1}^m \frac{1}{c(i)} \sum_{j=t(i)}^{t(i+1)-1}  W^*_{i,j}\
S_i \ W_{i,j} \otimes 1_{c(i)}=A.$$ On the other hand, if
$\oplus_{i=1}^m A_i\otimes 1_{c(i)}=\mathcal E_\mathcal A(U^*BU)$ it
is clear that $[\lambda(A_i\otimes 1_{c(i)})]_{i=1}^m\in
M_B(\mathcal A)$. The second claim in \eqref{uit1} follows from
Lemma \ref{buena defi} and the previous arguments.

To prove \eqref{uit2}, we proceed in a similar way. We first
define $\mathcal D_B^w(\mathcal A)$ as the set containing all
$c$-tuples $(\mu^i)_{i=1}^c$ with $\mu^i\in (\RR_{\geq 0})^{k(i)}$,
$\mu^i=(\mu^i)^\downarrow$ for $1\leq i\leq c$ and such that the
$(c+1)$-tuple $$(\lambda(B), (\mu^1,0_{n-k(1)} ),\ldots,
(\mu^c,0_{n-k(c)}))$$ satisfies Klyachko's compatibility
inequalities plus $\tr(B)\geq \sum_{i=1}^c\sum_{j=1}^{k(i)}\mu^i_j$.
Then $N_B^w(\mathcal A)$ and $ M_B^w(\mathcal A)\subseteq (\RR_{\geq 0})^n$ are defined in
terms of $\mathcal D_B^w(\mathcal A)$ also using the conditions a)
and b). The interested reader can now check that $M_B^w(\mathcal A)$
has the desired properties following a similar argument to that
above.
\end{proof}

\begin{rem}\label{com con mayo usual}
Notice that in case $\mathcal A$ is the maximal abelian  subalgebra
of $\mat$ of (complex) diagonal matrices with respect to the
canonical basis, the set $M_B(\mathcal A)$ is already closed by
permutation for any $B\in \matsa$. That is, for every $\sigma\in
\mathbb S_n$ then $\lambda_\sigma=(\lambda_{\sigma(i)})_{i=1}^n\in
M_B(\mathcal A)$ if and only if $\lambda=(\lambda_i)_{i=1}^n\in
M_B(\mathcal A)$.  To see this last claim note that if $P_\sigma$ is
the permutation matrix associated with $\sigma\in \mathbb S_n$ and
$B\in \matsa$ then $$\mathcal E_\mathcal A(P_\sigma^*
U^*BUP_\sigma)=P_\sigma \mathcal E_\mathcal A(U^*BU) P_\sigma$$
where $\mathcal E_\mathcal A(U^*BU)$ is now a diagonal matrix.
\end{rem}

\begin{cor}\label{mayo es espectral}
Let $\mathbf l=(d(i),c(i))_{i=1}^m\in (\NN^2)^m$ be such that
$\sum_{i=1}^md(i)\cdot c(i)=n$. Using the notations of the
NC-Schur-Horn theorem we have
\begin{enumerate}
\item Given $A,\, B\in \matsa$, there exist unitary matrices
$U,\,V\in \U(n)$ such that $U^*AU=\mathcal E_\mathcal A(V^*BV)$ if
and only if $\lambda(A)\in M_B(\mathbf l):=\{\mu^\downarrow:\ \mu\in
M_B(\mathcal A)\}$.
\item Given $A,\,B\in \matpos$, there exist
$U,\,V\in \mat$ with $U\in \U(n)$ and $\|V\|\leq 1$ such that
$U^*AU=\mathcal E_\mathcal A(V^*BV)$ if and only if $\lambda(A)\in
M_B^w(\mathbf l)=\{ \mu^\downarrow:\ \mu\in M_B^w(\mathcal A)\}$.
\end{enumerate}
\end{cor}

\section{Extended majorization in $\matsa$}\label{noentiendo}

In this section, using the previous results, we present an spectral relation between selfadjoint matrices that extends majorization. For other extensions of majorization, the so called joint majorizations, see \cite{FPL}.  

\subsection{Definition of extended majorization and basic properties}

\begin{fed}[Extended majorization and submajorization]\label{la defi de mayo}
Let $A,\,B\in\matsa$ and let $\mathbf l=(d(i),c(i))_{i=1}^m\in
(\NN^2)^m$ such that $\sum_{i=1}^m d(i)\cdot c(i)=n$.
 We say that $B$ $\mathbf l$-majorizes $A$, denoted $A\prec_\mathbf l B$
iff
$$ \U_n(A)\cap \mathcal E_\mathcal A(\U_n(B))\neq \emptyset \ \ \text{ or equivalently } \ \ A\in \U_n( \mathcal E_\mathcal A(\U_n(B)))$$
for any (and then every) unital $*$-subalgebra $\mathcal A\subseteq
\mat$ with spectral list $\mathbf l$.


If we further assume that $A,\, B\in \matpos$ then we say that $B$
$\mathbf l$-submajorizes $A$, denoted $A\prec_{\mathbf l, \, w}B$
iff
 $$\U_n(A)\cap  \mathcal E_\mathcal A(\C_n(B))\neq \emptyset \ \ \text{ or
equivalently } \ \ A\in \U_n( \mathcal E_\mathcal A(\C_n(B)))$$ for
any (and then every) unital $*$-subalgebra $\mathcal A\subseteq
\mat$ with spectral list $\mathbf l$.

\end{fed}

Note that Lemma \ref{buena defi} is the statement that $\mathbf
l$-majorization and $\mathbf l$-submajorization are actually well defined

It is implicit in Definition \ref{la defi de mayo} that these notions are actually well defined  up to equivalence of spectral lists: given $A,\,B\in \matsa$ (resp
$A,\,B\in \matpos$) and $\mathbf l=(d(i),c(i))_{i=1}^m$ with
$\sum_{i=1}^m d(i)\,c(i)=n$ then $A\prec _\mathbf l B$ if and only
if $A\prec_{\mathbf l(\sigma)}B$ (resp $A\prec _{\mathbf l,w} B$ if
and only if $A\prec_{\mathbf l(\sigma),w}B$) for any (every)
$\sigma\in \mathbb S_m$, where $\mathbf l(\sigma)=(
d(\sigma(i)),c(\sigma(i)))_{i=1}^m$.

\begin{rem}\label{mayo depende de Klyachko}
As a consequence of the NC-Schur-Horn theorem and Corollary
\ref{mayo es espectral} we conclude that $\mathbf
l$-(sub)majorization is an spectral relation that can be described
explicitly in terms of Klyachko's compatibility inequalities. On the
other hand, majorization in $\matsa$ in the sense of Ando
corresponds to $\mathbf l$-majorization for the list $\mathbf
l=(1,1)_{i=1}^n$ and hence the $\mathbf l$-majorization is an
extension of usual majorization.\end{rem}

We shall need the following notion of refinement between
multiplicity free lists. Given $\mathbf l_1=(d_1(i),1)_{i=1}^m$,
$\mathbf l_2=(d_2(i),1)_{i=1}^t$ such that $\sum_{i=1}^m d_1(i)=\sum_{i=1}^t d_2(i)=n$ we say that $\mathbf l_1$ refines
$\mathbf l_2$ if there exist unital $*$-subalgebras $\mathcal
A\subseteq \mathcal B\subseteq \mat$ such that $\mathcal A$ has
spectral list $\mathbf l_1$ and $\mathcal B$ has spectral list
$\mathbf l_2$. It is clear that  $\mathbf l_1$ refines $\mathbf l_2$
if and only if there exists a partition $\{ D(i)\}_{i=1}^ t$ of the
set $\{1,\ldots,m\}$ such that $$ \sum_{i\in D(k)} d_1(i) =d_2(k) ,
\ \ 1\leq k\leq t.$$ Notice that every  multiplicity free list is
refined by the spectral list of a maximal abelian $*$-subalgebra of
$\mat$.

\begin{pro}\label{cororo} Let
$\mathbf l_1=(d_1(i),1)_{i=1}^m$, $\mathbf l_2=(d_2(i),1)_{i=1}^t$
be multiplicity free lists such that $\sum_{i=1}^m
d_1(i)=\sum_{i=1}^t d_2(i)=n$. If we assume that $\mathbf l_1$
refines $\mathbf l_2$ then $\mathbf l_2$-(sub)majorization implies
$\mathbf l_1$-(sub)majorization. In particular, $\mathbf
l_1$-(sub)majorization is a reflexive and antisymmetric relation modulo unitary
equivalence.
\end{pro}

\begin{proof}
Let $\mathcal A\subseteq \mathcal B\subseteq \mat$ be unital
$*$-subalgebras with spectral lists $\mathbf l_1$ and $\mathbf l_2$
respectively. Let $\mathcal E_\mathcal A$, $\mathcal E_\mathcal B$
be the corresponding TCE onto $\mathcal A$ and $\mathcal B$. Notice
that in this case we have $\mathcal E_\mathcal A \circ \mathcal
E_\mathcal B=\mathcal E_\mathcal A$. Moreover, since $\mathcal A$ is
multiplicity free then there exists a maximal abelian $*$-subalgebra
of $\mat$, denoted by $\mathcal D$, such that $\mathcal D\subseteq
\mathcal A$. If we denote by $\mathcal E_\mathcal D$ the TCE onto
$\mathcal D$ then $\mathcal E_\mathcal A \circ \mathcal E_\mathcal D
 =\mathcal E_\mathcal D$.

Let $A,\,B\in \matsa$ and let $U,\,V\in \U(n)$ be such that
$\mathcal E_\mathcal B(U^*B\,U)=V^*A\,V$. Without loss of
generality, we can assume that $V^*AV\in\mathcal D$. Then, $$
\mathcal E_\mathcal A\circ \mathcal E_\mathcal B( U^*B\,U)=\mathcal
E_\mathcal A(V^*A\,V)=V^*A\,V$$ since $V^*A\,V\in \mathcal D$. A
similar argument shows the submajorization statement. As a
consequence of the argument above, we conclude that $\mathbf
l_1$-(sub)majorization implies $\mathbf l$-(sub)majorization, where
$\mathbf l=(1,1)_{i=1}^n$ i.e. usual majorization. This last fact
implies the antisymmetry of $\mathbf l_1$-(sub)majorization.
\end{proof}

It is clear that $\mathbf l$-(sub)majorization, for lists which are
not multiplicity free, is not reflexive nor antisymmetric in
general. Because of these facts, in what follows we shall focus
$\mathbf l$-(sub)majorization for multiplicity free lists. On the
other hand, the question of transitivity of $\mathbf
l$-(sub)majorization for a general list is open.

\subsection{Extended submajorization and convex functions}

Given $A,\,B\in \matsa$ we say that $A$ is
\emph{spectrally dominated} by $B$, denoted $A\lesssim B$, if
$\lambda(B)_i\geq \lambda(A)_i$ for $1\leq i \leq n$. In this
context it is straightforward that, given $A,\,B\in \matpos$ then
$A\lesssim B$ if and only if there exists a contraction $V\in \mat$
such that $V^*BV=A$; but note that this last equation is, by
definition, $A\prec_{\mathbf t,\,w} B$ for the trivial list $\mathbf
t=(n,1)$. Any multiplicity free list $\mathbf l=(d(i),1)_{i=1}^m$
with $\sum_{i=1}^m d(i)=n$, refines (as defined before Proposition
\ref{cororo}) the list $\mathbf t$. Hence, by Proposition
\ref{cororo}, we get the equivalence: for $A,\,B\in \matpos$,
\begin{equation}\label{hecho import}A\lesssim B
 \ \text{ if and only if }\  A\prec_{\mathbf l,\,w} B\end{equation}
 for every multiplicity free list $\mathbf l$ as above.

\begin{pro}[Jensen's inequality]\label{jensen ineq}
Let $f:(\alpha,\beta)\rightarrow [0,\infty)$ be a monotone convex
function and let $A\in \matsa$ be such that the spectrum of $A$ is
contained in $(\alpha,\beta)$. Then, for every system of coordinate
projections $\cP=\{P_i\}_{i=1}^t$ and for every multiplicity free
list $\mathbf l=(d(i),1)_{i=1}^m$ with $\sum_{i=1}^md(i)=n$ $$
f(\C_\cP(A))\prec_{\mathbf l,\,w} \C_\cP(f(A)) .$$
\end{pro}

\begin{proof}
Let $f$ be a monotone convex function and let $\mathcal P$  be a
system of coordinate projections as above. As a consequence of
theorem 3.1 in \cite{JD} we get that $f(\C_\cP(A))\lesssim
\C_\cP(f(A))$. The result now follows from \eqref{hecho import}.
\end{proof}

Let $f:[0,\infty)\rightarrow [0,\infty)$ be a convex function with $f(0)=0$,  and
hence non-decreasing. If $A,\,B\in \matpos$ are such that $A\prec_w B$ in
the sense of Ando, Theorem \ref{prelimcsh} implies that $f(A)\prec_w
f(B)$ i.e. $f$ is monotonic with respect to submajorization.
Indeed, if $g:[0,\infty)\rightarrow \RR$ is an arbitrary
non-decreasing convex function then $g\circ f$ is again
non-decreasing and convex. Therefore by hypothesis and
Theorem \ref{prelimcsh} we have $$\tr(g(f(A)))=\tr( g\circ f(A))\leq
\tr(g\circ f(B))=\tr(g(f(B))).$$ Since $g$ was arbitrary, again by
Theorem \ref{prelimcsh}, we get that $f(A)\prec_w f(B)$. The next
result is a generalization of this fact to the context of $\mathbf
l$-submajorization for multiplicity free lists $\mathbf l$.
\begin{pro}\label{monotonia} Let $f:[0,\infty)\rightarrow [0,\infty)$ be a convex
function with $f(0)=0$ and let $\cP=\{P_i\}_{i=1}^m$ be a system of coordinate projections with rank$(P_i)=d(i)$, $1\leq i\leq m$. Let $V\in\mat$ be such that $\|W\|\leq 1$ and $A,\,B\in\matpos$ be such that $\cC_\cP(W^* B \,W)=A$. Then, there exists $\tilde W\in \mat$ with $\|\tilde W\|\leq 1$ and such that $\cC_\cP(\tilde{W} ^* f(B)\,\tilde W)=f(A)$.
\end{pro}
\begin{proof}
Let $A,\,B\in \matpos$ be such that $A\prec_{\mathbf l,\,w} B$ and
let $f$ be as above. We assume that
$\C_\cP(W^*BW)=\oplus_{i=1}^m A_i=A$ for a contraction $W\in \mat$.
 We shall need the following result from \cite{sil}: if $X\in\mat$ is a
contraction then there exists $V\in \U(n)$ such that $f(X^*BX)\leq
V^*X^*f(B)XV$. Fix $1\leq i\leq m$. Using the previous result and
fact that for every $T\in \mat$ then $TT^*$ and $T^*T$ are unitarily
equivalent, we conclude that there there exist $U_i,\,V_i\in
\U(n)$ such that
$$f((P_i W^*)B(W  P_i))\leq V_i^* P_i W^*f(B)W P_i V_i = U^*_i f(B)^{1/2}W
P_iW^* f(B)^{1/2} U_i.$$ Then \begin{eqnarray*}f(B)&\geq&
\sum_{i=1}^m f(B)^{1/2}W P_i W^* f(B)^{1/2}\\ &\geq& \sum_{i=1}^m
U_i f(P_i W^*BW P_i) U_i^*=\sum_{i=1}^m U_i (\oplus_{j=1}^m
\delta_{ij} \, f(A_i)) U_i^*\end{eqnarray*} and the proposition now
follows from Theorem \ref{teo contrac sh}.
\end{proof}

\begin{cor}
Let $\mathbf l=((d(i),1))_{i=1}^m$ be a multiplicity free list with $\sum_{i=1}^m d(i)=n$. If $A,\,B\in \matpos$ then the following statements are equivalent:
\begin{enumerate}
\item $A\prec_{\mathbf l,w}B$.
\item For every convex function $f:[0,\infty)\rightarrow [0,\infty)$ with $f(0)=0$  we have
$f(A)\prec_{\mathbf l,\,w}f(B)$.
\end{enumerate} In particular, every convex function $f:[0,\infty)\rightarrow [0,\infty)$ with $f(0)=0$ is monotonic with respect to $\mathbf l$-submajorization.
\end{cor}

The next result, which follows form our previous arguments, is
theorem 2.1 in \cite{bou} expressed in terms of convex functions.
Its proof illustrates the use of extended majorization.
\begin{cor}\label{coro bou}
Let $f:[0,\infty)\rightarrow[0,\infty)$ be a convex function with
$f(0)=0$ and let $A,\, B\in \matpos$. Then there exist unitary
matrices $U,\,V\in \U(n)$ such that $$U^*f(A)U+ V^* f(B) V\leq
f(A+B).$$
\end{cor}
\begin{proof}
Consider the $2n\times 2n$ matrices \begin{equation}\label{eqmayo1}\begin{pmatrix} A+B & 0 \\
0 & 0\end{pmatrix}=\begin{pmatrix} A & 0 \\ 0 & 0\end{pmatrix}+
\begin{pmatrix} 0 & 1 \\ 1 & 0\end{pmatrix}\,
\begin{pmatrix} 0 & 0 \\ 0 &
B\end{pmatrix} \,\begin{pmatrix} 0 & 1 \\ 1 & 0\end{pmatrix}
\end{equation} Let $\mathbf l=((n,1),(n,1))$. By Theorem
\ref{teo contrac sh} and Proposition \ref{monotonia} there exist unitary matrices $\tilde U,\,\tilde
V\in \U(2n)$ such that
\begin{equation}\tilde U^* \begin{pmatrix} f(A) & 0 \\
0 & 0 \end{pmatrix}\tilde U+ \tilde V \begin{pmatrix} 0 & 0 \\
0 & f(B)\end{pmatrix}\tilde V \leq \label{cand}
 \begin{pmatrix} f(A+B) & 0 \\
0 & 0\end{pmatrix} .\end{equation} If $\tilde U=(U_{ij})_{ij=1}^2$
then, by compressing \eqref{cand} to the (2,2) block we get
$U^*_{12} f(A) U_{12}=0$ and hence $QU_{12}=0$ where $Q$ is the
projection onto the range of $f(A)\geq 0$. Therefore
$$U_{11}\,U^*_{11}+U_{12}\,U_{12}^*=1_n \ \Rightarrow Q(
U_{11}\,U^*_{11})Q=Q.$$ Thus, $QU_{11}\in \mat$ is a partial
isometry. If $U\in \U(n)$ is such that $QU=QU_{11}$ then
\begin{equation}\label{hay unita1}
U^*f(A)U=U_{11}^*\,f(A)U_{11}=\left(\tilde U  \begin{pmatrix} f(A) & 0 \\
0 & 0 \end{pmatrix} \tilde U \right)_{11}\end{equation} where the
sub-index $11$ in the right-hand side of this last equation stands
for the $(1,1)$-block. Similarly, there exists $V\in \U(n)$ such
that
\begin{equation}\label{hay unita2}V^*f(B)V=
\left(\tilde V  \begin{pmatrix} 0 & 0 \\
0 & f(B) \end{pmatrix} \tilde V\right)_{11}\end{equation} The
corollary now follows by compressing the inequality \eqref{cand} to
the $(1,1)$-block and using \eqref{hay unita1} and \eqref{hay
unita2}.
\end{proof}

\subsection{A non-commutative Horn's lemma and QIT}

In \cite{rusk} the following problem is posed in the context of
Quantum Information Theory (QIT).
\begin{conj}[from \cite{rusk}] \label{conj5} Let
$A\in\mathcal M_{d\cdot m}(\CC)^{+}$ be a block
matrix $A=(A_{ij})_{i,\,j=1}^m$ with $A_{ij}\in \mathcal M_d(\CC)$
for $1\leq i,\,j\leq m$ and let $M=\sum_{i=1}^m A_{ii}\in \mathcal
M_d(\CC)^+$. Then there exist rectangular matrices $X_i\in \mathcal
M_{d\cdot m,\,d}(\CC)$, $X_i^*=(X_{1\, i}^*,\ldots, X_{m\,i}^*)$
with $X_{i\,j}\in \mathcal M_{d}(\CC)$ for $1\leq i,\,j\leq m$ such
that
\begin{equation}\label{desc rusk}
 A=\frac{1}{m}\,\sum_{i=1}^m X_i X_i^* \ \ \text{ and } \ \
 \sum_{j=1}^m X_{j\,i} \, X_{j\,i}^*=M , \  1\leq i\leq m.
 \end{equation}
\end{conj}
The previous conjecture can be also expressed in terms of partial
traces. By the arguments in section \ref{lasec22}  we see that we
can replace \eqref{desc rusk} in Conjecture \ref{conj5} by
\begin{equation}\label{equiv conj5}
A=\frac{1}{m}\,\sum_{i=1}^m X_i X_i^* \ \ \text{ and } \ \
 \Tr_m(X_i\,X_i^*) = \Tr_m(A) , \  1\leq i\leq m.
\end{equation}

Conjecture \ref{conj5} is related with certain convex decompositions
of unital completely positive (UCP) maps between matrix algebras, in
terms of Choi matrices, that are of interest in QIT. The case $d=1$
is solved in \cite{rusk} using what is called ``Horn's lemma" namely,
that given $A\in \matpos$ with $\tr(A)=1$ there exist $U,\,B\in
\mat$ with $U$ unitary, $B$ with diagonal entries all equal to $1/n$
and $U^*A\,U=B$. The following result is an analogue of the above
Horn's lemma which leads to a related representation to that in
\eqref{desc rusk}. Still, while the convex decomposition that is
obtained using our result expresses a UCP map as an average of
completely positive maps with Choi rank at most $m$, these
representing maps may fail to be unital.

\begin{pro}[A non-commutative Horn's lemma]\label{Hlem}Let $A\in\mathcal
M_{d\cdot m}(\CC)^{+}$ be a block matrix $A=(A_{ij})_{i,\,j=1}^m$
with $A_{i\,j}\in \mathcal M_d(\CC)$ for $1\leq i,\,j\leq m$. Let
$\cP=\{P_i\}_{i=1}^m$ be a system of coordinate projections such
that rank$\,(P_i)=d$, $1\leq i\leq m$.
\begin{enumerate}
\item\label{hl1} There
exists $U\in \U(d\cdot m)$ and $D\in \mathcal M_{d}(\CC)^+$ such that
$$\C_\cP(U^*A\,U)=\frac{1}{m}\oplus_{i=1}^mD.$$
\item\label{hl2} There exist rectangular matrices $X_i\in \mathcal
M_{d\cdot m,\,d}(\CC)$, $X_i^*=(X_{1\, i}^*,\ldots, X_{m\,i}^*)$
with $X_{i\,j}\in \mathcal M_{d}(\CC)$ and unitary matrices $U_i\in
\U(d\cdot m)$  for $1\leq i,\,j\leq m$, such that
\begin{equation}\label{desc rusk weak1}
 A=\frac{1}{m} \,\sum_{i=1}^m X_i\, X_i^*
 \ \ \text{ and } \ \ U_i^* X_i\,X_i^*U_i=\oplus_{j=1}^m \delta_{ij}\, D\, ,
\end{equation}
\begin{equation}\label{desc rusk weak2}
  \ \ \text{ and hence } \ \ \Tr_m(U_i^* (X_iX_i^*)U_i)=\Tr_m(U^*AU),\ \   1\leq i\leq m.
 \end{equation}
\end{enumerate}
\end{pro}
\begin{proof}
Note that \eqref{hl2} is a direct consequence of \eqref{hl1} and
Theorem \ref{teo KSH} (with $\alpha=0$). Indeed, if we assume \eqref{hl1} then,
there exist $U_i\in \mathcal \U(d\cdot m)$ for $1\leq i\leq m$ such
that
$$A=\frac{1}{m}\sum_{i=1}^m U_i \,(\oplus_{j=1}^m \delta_{ij} \, D) \, U_i^*
=\frac{1}{m} \sum_{i=1}^m X_i\,X_i^*$$ where
$$X_i^*=(D^{1/2}(U^{(i)}_{1i})^*,\ldots,D^{1/2} (U^{(i)}_{mi})^*),$$
with $U_i=(U^{(i)}_{lk})_{l,k=1}^m$ and $U^{(i)}_{lk}\in \mathcal
M_d(\CC)$.

 To prove \eqref{hl1} consider first
$\xi\in \CC$ an $m$-th primitive root of unity and let $\tilde V\in
\mathcal M_m(\CC)$ be the matrix with $j$-th row given by
$$R_j(\tilde
V)=1/\sqrt{m}\,(1,\,\xi^j,\,\xi^{2j},\ldots,\xi^{(m-1)j})\,,\ \
1\leq j\leq m.$$ It is then straightforward to show that the rows of
$\tilde V$ form an orthonormal basis for $\CC^m$ and hence $\tilde
V\in \U(m)$ is a unitary matrix. Let $V\in \U(d\cdot m)$ be the
block matrix $V=(\tilde V_{ij}\cdot 1_d)_{i,j=1}^m$. If $W\in
\U(d\cdot m)$ is such that $W^*AW=\oplus_{i=1}^m D_i$  where
$D_i\in \mathcal M_{d}(\CC)$ is a diagonal matrix $1\leq i\leq m$,
define
$U:=VW$ and note that
$$ \C_\cP((WV)^*A(WV))=\C_\cP(V^*(\oplus_{i=1}^m
D_i)V)=\frac{1}{m}\, \oplus_{i=1}^m (\sum_{j=1}^m D_j)$$
The last equality follows from the block structure of
$\oplus_{i=1}^m D_i $ and by construction of $V$. Thus, we define
$D:=\sum_{j=1}^m D_j$.
\end{proof}
Note that the particular case $d=1$ of Conjecture \ref{conj5}
follows from Proposition \ref{Hlem}, since $\Tr_m=\tr$ in this case.
But we remark that the general case of Conjecture \ref{conj5} does
not follow from Proposition \ref{Hlem}, since the equation
$\Tr_m(U^*A\,U)=\Tr_m(U^*_i X_i\, X_i^* \,U_i)$ does not imply (for
$d> 1$) that $\Tr_m(X_iX_i^*)=\Tr_m(A)$. This is a consequence of
the non-commutativity of the values of $\Tr_m$.

 Also notice that the matrix $D$ above is not
unique. Moreover, there does not seem to be a canonical choice of
$D$ in general. Hence, if we let $\mathbf d=(d,\ldots,d)\in \RR^m$,
it is not clear whether there is in general a minimum (up to unitary
equivalence) with respect to $\mathbf d$-majorization of the set $\{
A\in \mathcal M_{d\cdot m}(\CC)^{+}:\ \Tr_m(A)=1\}.$

\begin{rem} It is worth noting that the case $m=2$ of the conjecture 5 in
\cite{rusk} has been proved (see \cite{rusk}). But the ideas
involved in the proof are related with the off-diagonal
blocks of the 2$\times$2 representation of $A$.\end{rem}

\smallskip

\noindent Pedro Massey \\ Departamento de Matem\'atica - FCE, UNLP, \\
Instituto Argentino de Matem\'atica, Argentina and\\
Dept. of Mathematics and Statistics - U of R, Regina, Canada.

\end{document}